\documentclass[a4paper,11pt
]{amsart}

\usepackage[latin1]{inputenc}
\usepackage[T1]{fontenc}
\usepackage[linktocpage]{hyperref} % For hyperlinks in the PDF [linktocpage]
	\hypersetup{colorlinks=true,linkcolor=blue,citecolor=blue,filecolor=magenta,urlcolor=blue}
\usepackage{geometry}       
\usepackage[english]{babel}  
\usepackage{amsmath,amsfonts,amssymb,amsthm}
\usepackage{dsfont}
\usepackage{tikz,tkz-graph,tikz-cd}

\newtheorem{thm}{Theorem}[section]

\newtheorem{propo}[thm]{Proposition}
\newtheorem{lem}[thm]{Lemma}
\newtheorem{cor}[thm]{Corollary}

\theoremstyle{definition}
\newtheorem{definition}[thm]{Definition}

\theoremstyle{remark}
\newtheorem{rmk}[thm]{Remark}

\newcommand{\CC}{\mathds{C}}
\newcommand{\ZZ}{\mathds{Z}}
\newcommand{\PP}{\mathds{P}}
\newcommand{\NN}{\mathds{N}}

\newcommand{\inv}{^{-1}}
\newcommand{\set}[1]{\left\{ #1 \right\}}
\newcommand{\lks}[1]{\LKS\left( #1 \right)}
\renewcommand{\hat}{\widehat}
\renewcommand{\tilde}{\widetilde}

\newcommand{\mcL}{\mathcal{L}}
\newcommand{\mcO}{\mathcal{O}}
\newcommand{\mcT}{\mathcal{T}}
\newcommand{\mcB}{\mathcal{B}}
\newcommand{\mcF}{\mathcal{M}}

\DeclareMathOperator{\HH}{H}
\DeclareMathOperator{\Stab}{Stab}

\DeclareMathOperator{\Aut}{Aut}
\DeclareMathOperator{\im}{Im}
\DeclareMathOperator{\Sing}{Sing}
\DeclareMathOperator{\Ind}{Ind}
\DeclareMathOperator{\LKS}{lks}
\DeclareMathOperator{\Pic}{Pic}

\begin{document}
%%%%%%%%%%%%%%%%%%%%%%%%%%%%%%%%%%%%%%%%%%%%%
%%%%%%%%%%%%%%%%%%%%%%%%%%%%%%%%%%%%%%%%%%%%%

\title{Non-homotopicity of the linking set of algebraic plane curves}

\author[B. Guerville-Ball\'e]{Beno\^it Guerville-Ball\'e}
        \address{Department of Mathematics, Tokyo Gakugei University, Koganei-shi, Tokyo 184-8501, Japan}
        \email{benoit.guerville-balle@math.cnrs.fr}
				\urladdr{www.benoit-guervilleballe.com}
\author[T. Shirane]{Taketo Shirane}
        \address{National Institute of Technology, Ube College, 2-14-1 Tokiwadai, Ube 755-8555, Yamaguchi Japan}
        \email{tshirane@ube-k.ac.jp}
				
		% Code AMS


\begin{abstract}
	The linking set is an invariant of algebraic plane curves introduced by Meilhan and the first author. It has been successfully used to detect several examples of Zariski pairs, i.e. curves with the same combinatorics and different embedding in $\CC\PP^2$. Differentiating Shimada's $\pi_1$-equivalent Zariski pair by the linking set, we prove, in the present paper, that this invariant is not determined by the fundamental group of the curve.
\end{abstract}

\maketitle

%%%%%%%%%%%%%%%%%%%%%%%%%%%%%%%%%%%%%%%%%%%%%
\section*{Introduction}
%%%%%%%%%%%%%%%%%%%%%%%%%%%%%%%%%%%%%%%%%%%%%

%Let $C$ be an algebraic plane curve of the complex projective plane $\PP^2:=\CC\PP^2$. 
In the present paper, we consider the homeomorphism class of a topological pair formed by an algebraic plane curve and the complex projective plane $\PP^2:=\CC\PP^2$ (referred by the \textit{embedded topology} of a plane curve in $\PP^2$). 
The embedded topology of a plane curve in its tubular neighborhood is called the \textit{combinatorics} of the curve (see \cite{ACT:survey} for details). It is known since Zariski~\cite{Zar:problem} that the combinatorics does not determined the embedded topology of the plane curve in $\PP^2$. A pair of two curves with a same combinatorics and different embedded topologies is called a \textit{Zariski pair}. There exists several methods to differentiate them; using, for example, Alexander polynomial~\cite{Oka:alexander,Zar:problem}, fundamental group of the complement~\cite{Shi:ZP,Ryb:fundamental,ArtCogGueMar:arithmetic}, double branched covers~\cite{Deg:deformation,Deg:construction,Shi:lattice,ArtTok:curve,Ban:splitting}, linking set~\cite{Gue:arithmetic,GueMei,GueViu} or the braid monodromy~\cite{ACCM:topology,ACM:braid}. This abundance of methods points the range of possible differences between topologies. For example, since the Alexander polynomial is determined by the fundamental group, we could say that the two curves of a Zariski pair are topologically further if they differ by their Alexander polynomials than if they are $\pi_1$-equivalent. It is thus important to understand whether an invariant determines another invariant, or not. 

In the present paper, we give an answer to the previous question for the fundamental group of the complement and the \emph{linking set} introduced by Meilhan and the first author in~\cite{GueMei}. Indeed, we prove that the linking set is not determined by the fundamental group of the curve. This result is obtained relating, in Theorem~\ref{thm:equivalence} and Proposition~\ref{propo:equivalence}, the linking set and the \emph{splitting numbers} introduced by the second author in~\cite{Shi}. The bridge between these two invariants is the \emph{curves linking} (Definition~\ref{def:linking} and Theorem~\ref{thm:linking}). It is an invariant sensing ``linking'' as the linking set (Theorem~\ref{thm:equiv_linking}) and formulated in algebro-geometric terms as the splitting numbers (Definition~\ref{def:linking}). The final step is the using of the second author result in~\cite{Shi}, stating that the splitting numbers detect the $\pi_1$-equivalent $k$-plets of Zariski given by Shimada in~\cite{Shi:equisingular} (Theorem~\ref{thm:shimada} and Proposition~\ref{propo:shirane}).

%%%%%%%%%%%%%%%%%%%%%%%%%%%%%%%%%%%%%%%%%%%%%
\section{Linking invariants}\label{sec:linking}
%%%%%%%%%%%%%%%%%%%%%%%%%%%%%%%%%%%%%%%%%%%%%

In the present section, we define the notion of \emph{curves linking}. Then we relate it with the linking set introduced by Meilhan and the first author.

\subsection{Curves linking}\mbox{}

Let $C=F+B$ be an algebraic plane curve such that 
$F$ and $B$ are of degree $f$ and $b$ respectively, and 
locally around each point $P\in F \cap B$, $F$ is smooth. For $P\in F \cap B$, we denote by $I_P$ the multiplicity of the intersection of $F$ and $B$ at $P$. We define $\sigma : \hat{\PP^2} \rightarrow \PP^2$ as a succession of blowing-ups over $F\cap B$ such that the strict transform of $C$ intersects transversally with the exceptional divisor of $\sigma$. Let $\tilde{F}$ (resp. $\tilde{B}$) be the strict transform of $F$ (resp. $B$) by $\sigma$; for $P\in F \cap B$, let $D_P$ be the irreducible component of $\sigma^{\ast}(B)$ which intersects with $\tilde{F}$ over $P$, where $\sigma^{\ast}(B)$ is the pull-back of $B$ as a divisor. We define $\hat{B}$ as the curve on $\hat{\PP^2}$ whose support is the sum of components of $\sigma^{\ast}(B)$ except $D_P$ for all $P\in F \cap B$: 
\[\hat{B}=\sigma\inv(B)-\sum_{P\in F\cap B}D_P, \] 
where $\sigma\inv(B)$ is the set-theoretic pre-image of $B$ and is regarded as a reduced divisor. By construction, $\tilde{F}\cap\hat{B}=\emptyset$, thus $\tilde{F} \subset \hat{\PP^2}\setminus\hat{B}$.

\begin{definition}\label{def:linking}
	\mbox{}
	\begin{enumerate}
		\item The \emph{curves linking} of $F$ with $B$ under $\sigma$ is the sub-group of $\HH_1(\hat{\PP^2}\setminus\hat{B})$ defined by:
	\begin{equation*}
		\mcL_{\sigma}(F,B) = \im \Big(\tilde{i'}_*:\HH_1(\tilde{F}) \rightarrow \HH_1(\hat{\PP^2}\setminus\hat{B})\Big),
	\end{equation*}
	where $\tilde{i'}_*$ is the map induced on the first homology group by the inclusion of $\tilde{F}$ in $\hat{\PP^2}\setminus\hat{B}$.
		\item The \emph{linking index} of $F$ with $B$ under $\sigma$, denoted by $\left[ B : F \right]_{\mcL_{\sigma}}$, is defined as the index of $\mcL_{\sigma}(F,B)$ in $\HH_1(\hat{\PP^2}\setminus\hat{B})$.
	\end{enumerate}
\end{definition}

\begin{rmk}
	\begin{enumerate}
	\item In the proof of Theorem~\ref{thm:linking}, we show that the curves linking $\mcL_{\sigma}(F, B)$ is isomorphic to the quotient of $\HH_1(F\setminus B)$ by its subgroup determined by intersection of $F$ and local branches of $B$. In particular, the curves linking $\mcL_{\sigma}(F, B)$ does not depend on the choice of $\sigma$ up to isomorphism. Hence we omit $\sigma$ when we mention the curves linking and the linking index; for example $\mcL(F, B)$ and~$[B:F]_{\mcL}$. 
	\item The curves linking of $F$ with $B$ is not equivalent to the curves linking of $B$ with~$F$.
	\end{enumerate}
\end{rmk}

\begin{thm}\label{thm:linking}
	Let $C_1=F_1+B_1$ and $C_2=F_2+B_2$ be as previously defined. If there is a homeomorphsim $h:\PP^2\rightarrow\PP^2$ such that $h(F_1)=F_2$ and $h(B_1)=B_2$, then $h$ induces an isomorphism $\hat{h}: \HH_1(\hat{\PP^2}\setminus\hat{B_1}) \rightarrow \HH_1(\hat{\PP^2}\setminus\hat{B_2})$ and we have:
	\begin{equation*}
		\hat{h}(\mcL(F_1,B_1)) = \mcL(F_2,B_2).
	\end{equation*}
\end{thm}

%
% Ben - There is a problem with the inclusion. It is not the good sense.
%
In order to prove this theorem, let us introduce some additional objects and notation together with two lemmas. Let $\sigma_B$ be the homeomorphism obtained by the restriction of $\sigma$ to $\hat{\PP^2}\setminus\sigma\inv(B)\rightarrow \PP^2\setminus B$. We denote by $\sigma_{B,*}$ the isomorphism induced by $\sigma_B$ on the first homology groups. By construction $\hat{\PP^2}\setminus\sigma\inv(B)$ is contained in $\hat{\PP^2}\setminus\hat{B}$, we denote by $j$ this inclusion and by $j_*$ the map induced on the first homology groups. Let $\Ind_F$ be the kernel of the map $j_*\circ\sigma_{B,*}\inv$. From this, we obtain the following isomorphism:
\begin{equation}\label{eq:isomorphism}
	\psi :  \HH_1(\PP^2\setminus B) / \Ind_F \longrightarrow \HH_1(\hat{\PP^2}\setminus\hat{B}).
\end{equation}
We denote by $\phi$ the inverse of $\psi$.

\begin{lem}\label{lem:meridian}
	The sub-group $\Ind_F$ is generated by:
	\begin{equation}\label{eq:meridians}
		m_P=\sum\limits_{b\in\mcB_P} I_{P,b} \cdot m_b,\quad \forall P \in F \cap B,
	\end{equation}
	where $\mcB_P$ is the set of local branches of $B$ at $P$, $I_{P,b}$ is the multiplicity of the intersection of $b$ with the local branch of $P$ contained in $F$, and $m_b$ is the meridian of the local branch~$b$.
\end{lem}

\begin{proof}
  It is clear that the kernel of the map $\HH_1(\PP^2\setminus B)\longrightarrow \HH_1(\hat{\PP^2}\setminus\hat{B}) $ is generated by the pre-images $m_P$ of the meridians of the exceptional divisors $D_P$ for $P\in F \cap B$. It is sufficient to prove that the $m_P$ are given by Equation~(\ref{eq:meridians}).
  
  Let $\mu_P$ be the meridian of the exceptional divisor $D_P$. We assume that $\mu_P$ is contained in $\tilde{F}$. The image of $\mu_P$ by $\sigma$ is the boundary $m_P$ of a disc centered in $P$ and contained in $F$. Let $S^3$ be the boundary of a 4-ball centered in $P$. Up to deformation, we can assume that $m_P$ is contained in $S^3$. We consider $L_P$ the link of $S^3$ defined as $S^3\cap C$. The component of $L_P$ associated to $F$ is unique since $F$ is locally smooth in $P$ and by construction it is $m_P$. In $S^3\setminus L_P$ we have that $m_P = \sum\limits_{b\in\mcB_P} \ell(m_P, c_b)m_b$, where $\ell(m_P, c_b)$ is the linking number of the component $m_P$ with $c_b$ the one associated with the local branch $b$ and $m_b$ is a meridian of $b$ contained in $S^3$. We conclude using the result of Brieskorn~\cite[Chapter~3, Proposition~13]{Brieskorn} stating that the linking number $ \ell(m_P, c_b)$ is equal to the multiplicity of the intersection of the corresponding local branches.
\end{proof}

\begin{lem}
	The map $\sigma$ naturally induces an isomorphism:
	\begin{equation*}
		\sigma_{F,*} : \HH_1(\tilde{F}\setminus \sigma\inv(B)) \longrightarrow \HH_1(F \setminus B).
	\end{equation*}
\end{lem}

\begin{proof}
Since $\sigma:\hat{\PP^2}\to\PP^2$ is a succession of blowing-ups over the intersection $F\cap B$, the restriction $\sigma_F:\tilde{F}\setminus\sigma\inv(B)\to F\setminus B$ of $\sigma$ is isomorphic. Thus the map $\sigma_F$ induces the isomorphism $\sigma_{F,\ast}:\HH_1(\tilde{F}\setminus\sigma\inv(B))\to \HH_1(F\setminus B)$. 
\end{proof}

\begin{proof}[Proof of Theorem~\ref{thm:linking}]
	We divide this proof in two steps. First we construct the map $\hat{h}$ then we prove the equality. By the way, to differentiate the maps associated with $C_1$ and the one with $C_2$, we add accordingly an exponent $1$ or $2$ to the denomination of these applications. The exponent is omitted if the proposition concerns both $C_1$ and $C_2$.\\
	
\noindent\textsc{Step 1.} Let $h_r:\PP^2\setminus B_1\to\PP^2\setminus B_2$ be the restriction of $h$ to the complement of $B_1$ and $B_2$. By hypothesis, this map is an homeomorphism too. Thus the map $h_{r,\ast}:\HH_1(\PP^2\setminus B_1)\to\HH_1(\PP^2\setminus B_2)$ induced by $h_r$ is an isomorphism. The map $h$ send meridians of $B_1$ to meridians of $B_2$, then $h_r$ too. Thus, Lemma~\ref{lem:meridian} implies that $h_r(\Ind_{F_1})=\Ind_{F_2}$. Hence the map $h_{r,\ast}$ induces an isomorphism $h_{r,\ast}':\HH_1(\PP^2\setminus B_1)/\Ind_{F_1}\to \HH_1(\PP^2\setminus B_2)/\Ind_{F_2}$. Using the isomorphism~(\ref{eq:isomorphism}), we define $\hat{h}$ as $\psi^2\circ h_{r,\ast}' \circ \phi^1$ (or equivalently $(\phi^2)\inv\circ h_{r,\ast}' \circ \phi^1$).\\
	
\noindent\textsc{Step 2.} To obtain the equality, consider the following diagram:
	\begin{equation*}
		\begin{tikzcd}
			\HH_1(\tilde{F}) \arrow{r}{\tilde{i'}_*} & 
			\HH_1(\hat{\PP^2}\setminus\hat{B}) \arrow[rightarrow, start anchor=east, end anchor=north west, "\phi"']{rd}{\simeq}& 
			\\
			\HH_1(\tilde{F}\setminus\sigma\inv(B)) \arrow[twoheadrightarrow]{u}{j_{F,\ast}}\arrow{r}{\tilde{i}_*}\arrow[rightarrow, "\simeq"']{d}{\sigma_{F,*}} & 
			\HH_1(\hat{\PP^2}\setminus\sigma\inv(B))\arrow[twoheadrightarrow]{u}{j_*}\arrow[rightarrow, "\simeq"']{d}{\sigma_{B,*}} & 
			\HH_1(\PP^2\setminus B) / \Ind_F 
			\\
			\HH_1(F\setminus B)\arrow{r}{i_*} & 
			\HH_1(\PP^2\setminus B) \arrow[start anchor=east, end anchor=south west]{ru}{\pi} &
		\end{tikzcd}
	\end{equation*}
	Here, $j_{F,\ast}$, $\tilde{i}_{\ast}$ and $i_\ast$ are the maps induced by the inclusions $\tilde{F}\setminus\sigma\inv(B)\to\tilde{F}$, $\tilde{F}\setminus\sigma\inv(B)\to\hat{\PP^2}\setminus\sigma\inv(B)$ and $F\setminus B\to\PP^2\setminus B$, respectively.  
	By construction, the three sqares are commutative. This implies that:
	\begin{equation*}
		\im(\tilde{i'}_*) = \im(\tilde{i'}_\ast\circ j_{F,\ast}\circ\sigma_{F,\ast}\inv) = \im(\phi\inv\circ\pi \circ i_*).
	\end{equation*}
	Moreover, by the construction of $h_{r,\ast}'$, we have $h_{r,\ast}'\circ\pi^1\circ i_{\ast}^1=\pi^2\circ i_{\ast}^2\circ h_{F,\ast}$, where $h_{F,\ast}:\HH_1(F_1\setminus B_1)\to\HH_1(F_2\setminus B_2)$ is the isomorphism induced by the restriction of $h$ to $F_1\setminus B_1$. We obtain that:
	\begin{align*}
		\hat{h}(\mcL(F_1,B_1))
		& = (\phi^2)\inv \circ h_{r,\ast}' \circ \phi^1  \left(\im(\tilde{i'}^1_*)\right) \\
		& = \im\left((\phi^2)\inv \circ h_{r,\ast}' \circ\pi^1\circ i^1_\ast\right) \\
		& = \im\left((\phi^2)\inv \circ \pi^2 \circ i^2_*\circ h_{F,\ast}\right) \\
		& = \im\left(\tilde{i'}^2_*\right) = \mcL(F_2,B_2).
		\qedhere
	\end{align*}
\end{proof}

\begin{cor}\label{cor:linking_index}
	If $C_1=F_1+B_1$ and $C_2=F_2+B_2$ are two homeomorphic curves as in Theorem~\ref{thm:linking}, then the linking index of $F_1$ with $B_1$ is equal to the one of $F_2$ with $B_2$.
\end{cor}

\subsection{Relation with the linking set}\mbox{}

In~\cite{GueMei}, Meilhan and the first author introduce another linking invariant of algebraic plane curve: the \emph{linking set}. In some particular cases, this invariant and the curve linking are equivalents.\\

One of the condition needed to obtain the equivalence is: 
\begin{equation}\label{eq:cond1}\tag{C.1}
	F \text{ is irreducible or } \HH_1(F)\simeq\ZZ. 
\end{equation}
We thus introduce the linking set up to this hypothese. The definition here given come from~\cite[Proposition~3.11]{GueMei} and Lemma~\ref{lem:meridian}. Let $\gamma$ be the image of an embedding of $S^1$ in $F$, such that $\gamma$ is not homologically trivial in $F$ and $\gamma \cap \Sing(C) = \emptyset$. 

\begin{definition}\label{def:linking_set}
	The \emph{linking set} of $\gamma$ is given by:
	\begin{equation*}
		\lks{\gamma} = \tilde{j}_\ast\left( \Big\{ \sum\limits_{i=1}^l a_i\cdot g_i \mid (a_1,\dots,a_l)\in\ZZ^l\setminus \set{(0,\dots,0)} \Big\} \right) \subset \HH_1(\PP^2\setminus B) / \Ind_F,
	\end{equation*}
	where $\set{g_1,\dots,g_l}$ is a basis of cycles of $\HH_1(F)$ contained in $F\setminus B$; and $\tilde{j}_\ast$ is the map induced by the inclusion of $F\setminus B$ in $\PP^2\setminus B$.
\end{definition}

The second needed condition to obtain the equivalence is:
\begin{equation}\label{eq:cond2}\tag{C.2}
	\HH_1(\PP^2\setminus B) / \Ind_F \text{ is a finite group}. 
\end{equation}
It is, for example, the case if $B$ is smooth. With this additional condition, the linking set always contains the trivial element. That is:
\begin{equation*}
	\lks{\gamma} = \tilde{j}_\ast\left( \Big\{ \sum\limits_{i=1}^l a_i\cdot g_i \mid (a_1,\dots,a_l)\in\ZZ^l \Big\} \right).
\end{equation*}
Thereby, the linking set does not depend on the cycle $\gamma$ but only on the curve $F$ containing $\gamma$. Thus, we can denoted it by $\lks{F}$. Furthermore, if we add a combinatorial condition on the plane curve $C=F+B$, then~\cite[Theorem 3.13]{GueMei} can be improved as follows:

\begin{thm}[\cite{GueMei}]
	Let $C_1=F_1+B_1$ and $C_2=F_2+B_2$ be two curves verifying the conditions~(\ref{eq:cond1}) and~(\ref{eq:cond2}), with the same embedded topology. Assume that any automorphism of the combinatorics of $C_i$ fixes $F_i$ and $B_i$ (i.e. any homeomorphism $h:(T(C_i),C_i)\to (T(C_i),C_i)$ satisfies $h(F_i)=F_i$ and $h(B_i)=B_i$, where $T(C_i)$ is a tubular neighborhood of $C_i$). Then we have:
	\begin{equation*}
		\lks{F_1} = \lks{F_2}.
	\end{equation*}
\end{thm}

\begin{proof}
	By~\cite[Corollary 4.6]{GueMei}, we can remove the oriented hypothesis of~\cite[Theorem 3.13]{GueMei}; the ordered hypothesis can be removed because all the automorphisms of the combinatorics of $C_i$ fix $F_i$ and $B_i$. 
\end{proof}

\begin{thm}\label{thm:equiv_linking}
	Let $C=F+B$ be a curve verifying conditions~(\ref{eq:cond1}) and~(\ref{eq:cond2}), we have:
	\begin{equation*}
		\lks{F} = \phi(\mcL(F,B)),
	\end{equation*}
	where $\phi$ is the isomorphism from $\HH_1(\hat{\PP^2}\setminus\hat{B})$ to $\HH_1(\PP^2\setminus B)/\Ind_F$.
\end{thm}

\begin{proof}
	In this proof, we use the notation introduced for the construction of the curves linking and we will prove that:
	\begin{equation*}
		\lks{F} = \im(\pi \circ i_*).
	\end{equation*}
	The group $\HH_1(F\setminus B)$ is generated by two kinds of elements. The first are the $g_i$'s which also generate $\HH_1(F)$, the second are the meridians contained in $F$ around the points of $F \cap B$. But Lemma~\ref{lem:meridian} implies that the second elements are in the kernel of the map $\pi$. To conclude, we use the diagram given in the proof of Theorem~\ref{thm:linking}.
\end{proof}

%%%%%%%%%%%%%%%%%%%%%%%%%%%%%%%%%%%%%%%%%%%%%
\section{Splitting numbers}\label{sec:splitting}
%%%%%%%%%%%%%%%%%%%%%%%%%%%%%%%%%%%%%%%%%%%%%

In this section, we recall the construction of the splitting numbers introduced by the second author in~\cite{Shi}. Then, we relate this invariant with the curves linking defined in Section~\ref{sec:linking}.

\subsection{The splitting numbers}\mbox{}

Let $\Psi:X\to \PP^2$ be a cyclic cover of degree $m$ branched along $B$ given by a surjection $\theta:\pi_1(\PP^2\setminus B)\twoheadrightarrow\ZZ/m\ZZ$. Let $B_{\theta}$ be the following divisor on $\PP^2$ 
\[ 
	B_{\theta}=\sum_{i=1}^{m-1}i\cdot B_i, 
\]
where $B_i$ is the sum of irreducible components of $B$ whose meridians are sent to $[i]\in\ZZ/m\ZZ$ by $\theta$. Note that the degree of $B_{\theta}$ is divided by $m$ since $\Psi$ is of degree $m$, say $\deg B_{\theta}=mn$. 

If $C=F+B$ is such that $F$ is an irreducible component, we define the \textit{splitting number} $s_{\Psi}(F)$ of $F$ for $\Psi$ as the number of irreducible components of $\Psi^{\ast}F$. In the following theorem, we give a way to compute it.

\begin{thm}\label{thm:splittingnumber}
	For a smooth curve $F\subset\PP^2$ which is not a component of $B$, the followings are equivalent: 
	\begin{enumerate}
		\item $s_{\Psi}(F)=\nu$; 
		\item $\nu$ is the maximal divisor of $m$ such that there exist a divisor $D$ on $\PP^2$ with $\nu D|_F= B_{\theta}|_F$. 
	\end{enumerate}
\end{thm}

In order to prove this theorem, we use the notation $\sigma:\hat{\PP^2}\to\PP^2$, $\tilde{F}$, $\tilde{B}$, $D_P$ and so on defined in Section~\ref{sec:linking}, and we need the following lemma:

\begin{lem}\label{lem:multiplicity}
	The multiplicity of $D_P$ in $\sigma^\ast ( B_{\theta})$ is $I_P$ for each $P\in F\cap B$, where $I_P$ is the intersection multiplicity of $F$ and $B_{\theta}$ at $P$. 
\end{lem}

\begin{proof}
	Let $\sigma_P:\hat{\PP^2}_P \rightarrow \PP^2$ be the blow-up in $\sigma$ over $P$. We denote $\tilde{F_P}$ the strict transform of $F$ by $\sigma_P$. Let $D'_P$ be the irreducible component of $\sigma_P^{\ast} (B_{\theta})$ which intersects with $\tilde{F_P}$. Let $n_P$ be the multiplicity of $D'_P$ in $\sigma_P^\ast (B_{\theta})$. It is sufficient to prove $n_P=I_P$ since the multiplicity of $D_P$ in $\sigma^{\ast}(B_{\theta})$ is equal to $n_P$. By the projection formula, we have:
	\begin{align*}
		n_P & = \sigma_P^\ast( B_{\theta}) \cdot \tilde{F_P} - \sum\limits_{Q\neq P} I_Q \\
		& =  B_{\theta}\cdot F - \sum\limits_{Q\neq P}I_Q \\
		& = I_P		
	 \qedhere
	\end{align*}
\end{proof}

\begin{proof}[Proof of Theorem~\ref{thm:splittingnumber}]
	It is enough to prove that the condition (2) implies the condition (1). Suppose the condition (2). Let $p_n:\mcT_{\mcO(n)}\to\PP^2$ be the line bundle associated to $\mcO_{\PP^2}(n)$, and let $X'$ be the subvariety of $\mcT_{\mcO(n)}$ defined by $t^m=p_n^{\ast}f_{ B_{\theta}}$, where $t\in H^0(\mcT_{\mcO(n)},p_n^{\ast}\mcO_{\PP^2}(n))$ is a tautological section, and $f_{ B_{\theta}}\in H^0(\PP^2,\mcO_{\PP^2}(mn))$ is a global section defining $ B_{\theta}$. Then $X$ is the normalization of $X'$, and $X$ and $X'$ are isomorphic over $\PP^2\setminus B$. Let $\varphi:X\to X'$ denote the normalization, and let $\Psi'$ be the restriction of $p_n$ to $X'$. Note that $\varphi$ satisfies $\Psi=\Psi'\circ\varphi$. 

	 We may assume that each irreducible component of $\sigma^{-1}(B)$ intersects with $\tilde{F}$ at most one point after more blowing-ups if necessary. Let $\sigma_F:\tilde{F}\to F$ be the restriction of $\sigma$ to $\tilde{F}$. For $P\in F$, let $\tilde{P}$ denote the point $\sigma_F^{\ast}(P)$ on $\tilde{F}$. By Lemma~\ref{lem:multiplicity}, we have $\mathrm{I}_P( B_{\theta},F)=\mathrm{I}_{\tilde{P}}(\sigma^{\ast} B_{\theta},\tilde{F})$ for $P\in F \cap B$, where $\mathrm{I}_P(B,F)$ is the local intersection number of the two divisors $B$ and $F$ at $P$. Let $\mcF$ be the following invertible sheaf on $\hat{\PP^2}$: 
	\[ 
		\mcF:=\mcO_{\hat{\PP^2}}\left(\sigma^{\ast}(n L)-\sum_{P\in F \cap B}\left[\frac{I_P}{m}\right] D_P\right), 
	\]
	where $L$ is a line on $\PP^2$, $I_P=\mathrm{I}_P( B_{\theta},F)$, and $[r]$ is the integer not beyond $r$ for a real number $r$. Let $p_{\mcF}:\mcT_{\mcF}\to\hat{\PP^2}$ be the line bundle associated to $\mcF$, and let $\hat{X}'$ be the subvariety of $\mcT_{\mcF}$ defined by $t_{\mcF}^m=p_{\mcF}^{\ast}f_{\hat{B}_{\theta}}$, where  $t_{\mcF}\in H^0(\mcT_{\mcF},p_{\mcF}^{\ast}\mcF)$ is a tautological section, and $f_{\hat{B}_{\theta}}\in H^0(\hat{\PP^2},\mcO(\hat{B}_{\theta}))$ is a section defining the following divisor $\hat{B}_{\theta}$ on $\hat{\PP^2}$: 
	\[ 
		\hat{B}_{\theta}:=\sigma^{\ast}( B_{\theta})-\sum_{P\in B\cap C}\left[\frac{I_P}{m}\right] m D_P. 
	\]
	We have $\mcF^{\otimes m}\cong \mcO_{\hat{\PP^2}}(\hat{B}_{\theta})$, and 
	\[ 
		\left(\mcF|_{\tilde{F}}\right)^{\otimes m}\cong \mcO_{\tilde{F}}\left(\sum_{P\in F \cap B}\left(I_P-\left[\frac{I_P}{m}\right]m\right)\tilde{P}\right). 
	\]
	Let $\hat{\varphi}:\hat{X}\to\hat{X}'$ be the normalization of $\hat{X}'$. Let $\hat{\Psi}':\hat{X}'\to \hat{\PP^2}$ be the restriction of $p_{\mcF}$ to $\hat{X}'$, 
	and let $\hat{\Psi}:\hat{X}\to\hat{\PP^2}$ be the composition $\hat{\Psi}'\circ\hat{\varphi}$. Then $\hat{\Psi}:\hat{X}\to\hat{\PP^2}$ is a cyclic cover of degree $m$ whose branch locus is the sum of irreducible components of $\sigma^{\ast}(B)$ except $D_P$ with $I_P\equiv 0 \pmod{m}$. By Stein factorization, we have a birational morphism $\tilde{\sigma}:\hat{X}\to X$ with $\Psi\circ\tilde{\sigma}=\sigma\circ\hat{\Psi}$. 
		\begin{equation*}
			\begin{tikzcd}[row sep=small]
				& X \arrow{ld}[swap]{\varphi} \arrow{lddd}{\Psi} & & \hat{X} \arrow{ld}[swap]{\hat{\varphi}} \arrow{ll}[swap]{\tilde{\sigma}} \arrow{lddd}{\hat{\Psi}} \\
				X' \arrow{dd}[swap]{\Psi'} & & \hat{X}' \arrow{dd}[swap]{\hat{\Psi}'}
				\\ \\ 
				\PP^2 & & \hat{\PP^2} \arrow{ll}{\sigma}
			\end{tikzcd}
		\end{equation*}
	Since the birational map $\tilde{\sigma}\circ\hat{\varphi}^{-1}$ gives an isomorphism from $\hat{X}'\setminus(\sigma\circ\hat{\Psi}')^{-1}(B)$ to $X\setminus\Psi^{-1}(B)$, 
	$\tilde{\sigma}\circ\hat{\varphi}^{-1}$ gives a one to one correspondence between the sets of irreducible components of $\Psi^{\ast}(F)$ and $(\hat{\Psi}')^{\ast}(\tilde{F})$ respectively. By \cite[Remark~2.3]{Shi}, to compute the number of the irreducible components of $(\hat{\Psi}')^{\ast}(\tilde{F})$, we may assume that $\hat{\Psi}'$ is essentially unramified over $\hat{F}$, i.e. $I_P \equiv 0 \pmod{m}$ for any $P\in F \cap B$. Hence we have $(\mcF|_{\tilde{F}})^{\otimes m}\cong\mcO_{\tilde{F}}$. Then the restriction of $\hat{\Psi}'$ to $(\hat{\Psi}')^{-1}(\tilde{F})$ is a unramified simple cyclic cover over $\tilde{F}$. The condition (2) implies that the order of $[\mcF|_{\tilde{F}}]\in\Pic^0(\tilde{F})$ is equal to $m/\nu$. By \cite[Lemma~2.6]{Shi}, the number of irreducible components of $(\hat{\Psi}')^{\ast}(\tilde{F})$ is equal to $\nu$. Therefore, we obtain $s_{\Psi}(F)=\nu$. 
\end{proof}

From the proof of Theorem~\ref{thm:splittingnumber}, we obtain the following Proposition. 

\begin{propo}\label{propo:splitting_components}
	The splitting number $s_\Psi(F)$ is equal to the number of connected components of~$\hat{\Psi}^*\tilde{F}$.
\end{propo}

\subsection{Relation between splitting numbers and curves linking}\mbox{}

In this subsection, we assume that $C=F+B$ is a plane curve such that $F$ is irreducible and $B$ is smooth. With such hypotheses, we have $\pi_1(\PP^2\setminus B)\cong\HH_1(\PP^2\setminus B)$, and there exists $m\in\NN^*$ such that $\HH_1({\PP^2}\setminus{B})/\Ind_F\simeq \ZZ_m$. Note that $m$ is equal to the $\gcd$ of all $I_P$ for $P\in B\cap F$ by Lemma~\ref{lem:meridian}. Furthermore, the curve $C$ verifies the conditions~\ref{eq:cond1} and~\ref{eq:cond2}, but also the condition needed for the definition of the splitting numbers.

Let $\Psi:X\to\PP^2$ be the $\ZZ_m$-cover branched at $B$, and let $\hat{\Psi}:\hat{X}\to\hat{\PP^2}$ be as in the proof of Theorem~\ref{thm:splittingnumber}. We denote by $\Phi:U\to\hat{\PP^2}\setminus\hat{B}$ the restriction of $\hat{\Psi}$ to $U:=\hat{\Psi}^{-1}(\hat{\PP^2}\setminus\hat{B})$. Note that $\Phi$ is unramified since $I_P\equiv0\pmod{m}$ for any $P\in B\cap F$. 
Since $\pi_1(\hat{\PP^2}\setminus\hat{B})$ is a quotient of $\pi_1(\PP^2\setminus B)$, $\pi_1(\hat{\PP^2}\setminus \hat{B})$ is abelian. Hence we obtain $\pi_1(\hat{\PP^2}\setminus\hat{B})\cong\ZZ_m$ since $\HH_1(\hat{\PP^2}\setminus\hat{B})\cong\HH_1(\PP^2\setminus B)/\Ind_F$. In particular, $\Phi$ is the universal cover of $\hat{\PP^2}\setminus\hat{B}$. 
%Let $\Psi:X_m \rightarrow \hat{\PP^2}\setminus\hat{B}$ be the universal cover of $\hat{\PP^2}\setminus\hat{B}$. Since $\HH_1(\hat{\PP^2}\setminus\hat{B})\cong\HH_1(\PP^2\setminus B)/\Ind_F$, we can easily see that $\Psi$ is a $\ZZ_m$-cover. We fix once for all $p$ a point of $\tilde{F}$ and $x\in \Psi\inv(p)$, and we denote by $\big(\Psi^*\tilde{F}\big)_0$ the connected component of $\Psi^* \tilde{F}$ containing $x$. 
Let $\Aut_{\Phi}(U)$ be the subgroup of automorphisms of $U$ which are compatible with $\Phi$, i.e. $\Aut_{\Phi}(U)=\{\rho\in\Aut(U)\mid \Phi\circ\rho=\Phi\}$. 
It is well-known that $\pi_1(\hat{\PP^2}\setminus\hat{B})$ and $\Aut_{\Phi}(U)$ are isomorphic. This isomorphism can be describe as follows. Let $c$ be an element of $\pi_1(\hat{\PP^2}\setminus\hat{B})$, $c$ is uniquely lifted in a path $\lambda_c$ of $U$ such that $\lambda_c(0)=x$. There is a unique element $\rho_c\in\Aut_{\Phi}(U)$ such that $\rho_c(x) = \lambda_c(1)$. The application $\mathcal{E}:\pi_1(\hat{\PP^2}\setminus\hat{B})\rightarrow \Aut_{\Phi}(U)$ defined by $\mathcal{E}(c)=\rho_c$ is this isomorphism.

\begin{lem}\label{lem:stab}
	Let $i$ be the inclusion of $\tilde{F}$ in $\hat{\PP^2}\setminus\hat{B}$. If $i_*$ is the induced application on the fundamental group, then:
	\begin{equation*}
		\Stab \big(\Phi^*\tilde{F}\big)_0  \simeq i_*(\pi_1(\tilde{F},p)).
	\end{equation*}
\end{lem}

\begin{proof}
	Let $c\in\pi_1(\hat{\PP^2}\setminus\hat{B},p)$ such that $c \subset \tilde{F}$. By connection, $\lambda_c$ is contained in $\big(\Phi^*\tilde{F}\big)_0$, then $\lambda_c(1) \in \big(\Phi^*\tilde{F}\big)_0$ and $\rho_c(x) \in \big(\Phi^*\tilde{F}\big)_0$. Since automorphisms in $\Aut_\Phi(U)$ fix globally $\Phi^*\tilde{F}$, if $\rho_c\big(\big(\Phi^*\tilde{F}\big)_0\big) \cap \big(\Phi^*\tilde{F}\big)_0 \neq \emptyset$, then $\rho_c\big(\big(\Phi^*\tilde{F}\big)_0\big) = \big(\Phi^*\tilde{F}\big)_0$ and $\rho_c \in \Stab \big(\Phi^*\tilde{F}\big)_0$. Thus, we have: 
	\begin{equation*}
		\mathcal{E}\circ i_*(\pi_1(\tilde{F},p)) \subset \Stab \big(\Phi^*\tilde{F}\big)_0.
	\end{equation*}

	Let $\rho\in\Stab \big(\Phi^*\tilde{F}\big)_0$ then $\rho(x)\in \big(\Phi^*\tilde{F}\big)_0$. Since $\big(\Phi^*\tilde{F}\big)_0$ is connected then it exists a path $\lambda\subset \big(\Phi^*\tilde{F}\big)_0$ such that $\lambda(0)=x$ and $\lambda(1)=\rho(x)$. We have then $\rho=\mathcal{E}\circ\Phi(\lambda)$. This implies that:
	\begin{equation*}
		\Stab \big(\Phi^*\tilde{F}\big)_0 \subset \mathcal{E}\circ i_*(\pi_1(\tilde{F},p)).
	\end{equation*}
	
	By double inclusion, we obtain that $\im_\mathcal{E}(i_*(\pi_1(\tilde{F},p))) = \Stab \big(\Phi^*\tilde{F}\big)_0$. Since $\mathcal{E}$ is an isomorphism, then the result holds.
\end{proof}

From this, we can deduce the following theorem:

\begin{thm}\label{thm:equivalence}
	Let $C=F+B$ be a curve such that $B$ and $F$ are smooth irreducible components of degree $b$ and $f$ respectively, with $b\neq f$. Then the following equation holds:
	\begin{equation*}
		s_\Psi(F) = \left[ B : F \right]_\mcL.
	\end{equation*}
\end{thm}

\begin{proof}
	Since $\pi_1(\hat{\PP^2}\setminus\hat{B},p)$ is abelian, then it is equivalent to consider the first homology group instead of the fundamental group. By the proof of Lemma~\ref{lem:stab}, $\Aut_{\Phi}(U)=\ZZ_m$ acts on the set $\mathrm{CC}{(\Phi^*\tilde{F})}$ of the connected components of $\Phi^*\tilde{F}$ transitively. Hence we have the following equation:
	\begin{equation*}
		\# \mathrm{CC}{(\Phi^*\tilde{F})}=[\ZZ_m : \Stab \big(\Phi^*\tilde{F}\big)_0]. 
	\end{equation*}
	Since $\Phi$ is the restriction of $\hat{\Psi}$ to $U:=\hat{\Psi}^{-1}(\hat{\PP^2}\setminus \hat{B})$, Proposition~\ref{propo:splitting_components} implies that $\# \mathrm{CC}{(\Phi^*\tilde{F})} = s_\Psi(F)$; and by Lemma~\ref{lem:stab}, we have that $\Stab \big(\Phi^*\tilde{F}\big)_0=i_*(\HH_1(\tilde{F}))$ since the stabilizer is contained in an abelian group. By hypothesis, $\HH_1(\hat{\PP^2}\setminus\hat{B})\simeq \ZZ_m$, then we obtain:
	\begin{equation*}
		[\ZZ_m : \Stab \big(\Phi^*\tilde{F}\big)_0]=[ \HH_1(\hat{\PP^2}\setminus\hat{B}) : i_*(\HH_1(\tilde{F})) ]. 
	\end{equation*}
	We conclude using the definition of the linking index given in Definition~\ref{def:linking}-(2).
\end{proof}

%%%%%%%%%%%%%%%%%%%%%%%%%%%%%%%%%%%%%%%%%%%%%
\section{The non-homotopicity of the linking set}\mbox{}
%%%%%%%%%%%%%%%%%%%%%%%%%%%%%%%%%%%%%%%%%%%%%

To prove the non-homotopicity of the linking set, we first recall the results of Shimada \cite{Shi:equisingular} and the second author \cite{Shi}. Then using the relations previously proven between the linking set and the curves linking together with the one between the curves linking and the splitting numbers, we obtain that the linking set detects Shimada's $\pi_1$-equivalent Zariski $k$-plets.

Let $\deg(B)=b\geq 3$. Since, by hypothesis, $m$ divises $b$, we define $n\in\NN^*$ such that $b=mn$.

\begin{definition}[{\cite[Definition~1.1]{Shi:equisingular}}]
	A projective plane curve $C\subset\PP^2$ is \textit{of type $(b,m)$} if it satisfies the following conditions; 
	\begin{enumerate}
	\item $C$ consists of two irreducible components $B$ and $F$ of degree $b$ and $3$, respectively, 
	\item $B$ and $F$ are smooth,
	\item the set $B \cap F$ consists of $3n$ points,
	\item the multiplicity of each point of $B \cap F$ is $m$.
	\end{enumerate}
\end{definition}

Let $\mathcal{F}_{b,m}$ be the family of all curves of type $(b,m)$. Using the notations and following the constructions of~\cite[Section~2]{Shi:equisingular}, $\mathcal{F}_{b,m}$ can be decomposed as:
\begin{equation*}
	\mathcal{F}_{b,m}=\coprod_{\mu | m}\mathcal{F}_{b,m}(\mu).
\end{equation*}
In \cite{Shi:equisingular}, Shimada proves the following theorem. 

\begin{thm}[{\cite[Theorem~1.2]{Shi:equisingular}}]\label{thm:shimada}
	Suppose that $b\geq 4$, and let $m$ be a divisor of $b$. 
	\begin{enumerate}
	\item The number of connected components of $\mathcal{F}_{b,m}$ is equal to the number of divisors of $m$.
	\item Let $C$ be a member of $\mathcal{F}_{b,m}$. Then the fundamental group $\pi_1(\PP^2\setminus C)$ is isomorphic to $\ZZ$ if $b$ is not divisible by $3$, while it is isomorphic to $\ZZ\oplus\ZZ/3\ZZ$ if $b$ is divisible by $3$. 
	\end{enumerate}	
\end{thm}

Hence the members of $\mathcal{F}_{b,m}$ can not be distinguished by the fundamental groups of their complements.  The following proposition is a consequence of \cite[Proposition~2.5]{Shi}, and it enables us to prove that the members of $\mathcal{F}_{b,m}$ provide a $\pi_1$-equivalent Zariski $k$-plet. 

\begin{propo}[{\cite[Proposition~2.5]{Shi}}]\label{propo:shirane}
Let $b$ and $m$ be positive integers such that $b\geq 3$ and $b\equiv 0\pmod{m}$. 
Let $\mu$ be a divisor of $m$, and let $C=B+F$ be a member of $\mathcal{F}_{b,m}(\mu)$. 
For the simple cyclic cover $\Psi:X_m\to\PP^2$ of degree $m$ branched along $b$, the splitting number $s_{\Psi}(F)$ of $F$ is equal to $m/\mu$. 
\end{propo}

As a consequence of Theorem~\ref{thm:equivalence} and Proposition~\ref{propo:shirane}, we obtain that the connected components of $\mathcal{F}_{b,m}$ can be distinguished using the linking index.

\begin{thm}\label{thm:index}
	Let $C=B+F$ be a curve of type $(b,m)$ with $b\geq 4$, contained in the connected component $\mathcal{F}_{b,m}(\mu)$ for a fixed divisor $\mu$ of $m$. The linking index of $F$ with $B$ is given by:
	\begin{equation*}
		[B,F]_\mcL = m/\mu.
	\end{equation*}
\end{thm}

\begin{propo}\label{propo:equivalence}
	Let $C=B+F$ be a curve such that $F$ and $B$ are smooth. The curves linking, the linking index and the linking set are equivalent.
\end{propo}

\begin{proof}
	If $B$ is smooth, then $\HH_1(\hat{\PP^2}\setminus\hat{B})\simeq \ZZ_m$. Since $\mcL(F,B)$ is a sub-group of a cyclic group, then it is determined by its index, that is the linking index. By definition, the linking index is determined by the curves linking. The equivalence between the linking set and the curves linking is given by Theorem~\ref{thm:equiv_linking}. Indeed, $B$ and $F$ are smooth, then they verify the condition~\ref{eq:cond2}. The condition~\ref{eq:cond1} is verified since $F$ is smooth.
\end{proof}

We conclude with the following corollary obtained from Proposition~\ref{propo:equivalence}, Theorem~\ref{thm:index} and Theorem~\ref{thm:shimada}.

\begin{cor}
	The curves linking, the linking index and the linking set are not determined by the fundamental group of the complement of a curve.
\end{cor}

%%%%%%%%%%%%%%%%%%%%%%%%%%%%%%%%%%%%%%%%%%%%%
%               Bibliography                %
%%%%%%%%%%%%%%%%%%%%%%%%%%%%%%%%%%%%%%%%%%%%%

%%%%%%%%%%%%%%%%%%%%%%%%%%%%%%%%%%%%%%%%%%%%%
%%%%%%%%%%%%%%%%%%%%%%%%%%%%%%%%%%%%%%%%%%%%%
\end{document}